\documentclass[11pt,reqno]{amsart}
\usepackage{amsmath,amssymb}
 \usepackage{cite}

 \makeatletter
 \evensidemargin  \oddsidemargin
 \makeatother

\newtheorem{theorem}{Theorem}
\theoremstyle{plain}

\newtheorem{corollary}{Corollary}

\newtheorem{example}{Example}

\numberwithin{equation}{section}

 \numberwithin{equation}{section}

\begin{document}

\title[Complementary inequalities to Davis-Choi-Jensen's inequality]
{Complementary inequalities to Davis-Choi-Jensen's inequality and operator power means}

\author[A. G. Ghazanfari]{A. G. Ghazanfari}

\address{Department of Mathematics, Lorestan University, P. O. Box 465, Khoramabad, Iran.}

\email{, ghazanfari.a@lu.ac.ir}

\setcounter{page}{1}

\date{}

 \begin{abstract}
 Let $f$ be an operator convex function on $(0,\infty)$, and $\Phi$ be a unital positive linear maps on $B(H)$.
 we give a complementary inequality to
 Davis-Choi-Jensen's inequality as follows
\begin{equation*}
f(\Phi(A))\geq \frac{4R(A,B)}{(1+R(A,B))^2}\Phi(f(A)),
\end{equation*}
where $R(A, B)=\max\{r(A^{-1}B) ,r(B^{-1}A)\}$ and $r(A)$ is the spectral radius of $A$.
We investigate the complementary inequalities related to the operator power means and the Karcher means via unital positive linear maps,
 and obtain the following result:
If $A_{1}$, $A_{2}$,\dots, $A_{n}$, are positive definite operators in $B(H)$, and $0<m_i\leq A_i\leq M_i$, then
\begin{equation*}
\Lambda( \omega;\Phi(\mathbb{A}))\geq\Phi(\Lambda( \omega; \mathbb{A}))\geq \frac{4\hbar}{(1+\hbar)^2}~\Lambda( \omega;\Phi(\mathbb{A})),
\end{equation*}
where $\hbar= \max\limits_{1\leq i\leq n} \frac{M_i}{m_i}$.
Finally, we prove that if $G(A_1,\dots,A_n)$ is the generalized geometric mean defined by Ando-Li-Mathias for $n$ positive definite operators, then
 \begin{align*}
\Phi(G(A_1,\dots,A_n))\geq\left(\frac{2h^\frac{1}{2}}{1+h}\right)^{n-1}G(\Phi(A_1),\dots,\Phi(A_n)),
\end{align*}
where $h=\max\limits_{1\leq i,j\leq n} R(A_i, A_j)$.
\end{abstract}

\maketitle

\section{\bf Introduction}\vskip 2mm

Let $B(H)$ denote the set of all bounded linear operators on a complex Hilbert
space $H$.
An operator $A \in B(H)$ is positive definite (resp. positive semi-definite)
if $\langle Ax, x\rangle > 0$ (resp. $\langle Ax, x\rangle \geq 0 )$ holds for all non‐zero $x \in H$ . If $A$ is
positive semi‐definite, we denote $A\geq 0$. Let $\mathcal{P}\mathcal{S}, \mathcal{P}\subset B(H)$ be the sets of all positive
semi-definite operators and positive definite operators, respectively.
To reach inequalities for bounded self-adjoint operators on Hilbert space, we shall use
the following monotonicity property for operator functions:\\
If $X\in B(H)$ is self adjoint with a spectrum $Sp(X)$, and $f,g$  are continuous real valued functions
on an interval containing $Sp(X)$, then
\begin{equation}\label{1.1}
f(t)\geq g(t),~t\in Sp(X)\Rightarrow ~f(X)\geq g(X).
\end{equation}
For more details about this property, the reader is referred to \cite{pec}.

For $A,B\in\mathcal{P}$ the geometric mean $A\sharp B$ of $A$ and $B$ is defined by
$A\sharp B =A^{\frac{1}{2}}(A^{\frac{-1}{2}}BA^{\frac{-1}{2}})^{\frac{1}{2}}A^{\frac{1}{2}}.$

For $A,B\in\mathcal{P}$, let
\[
R(A, B)=\max\{r(A^{-1}B) ,r(B^{-1}A)\}
\]
where $r(A)$ means the spectral radius of $A$ and we have
$$r(B^{-1}A)= \inf\{ \lambda >0 : A\leq \lambda B \}=\|B^{\frac{-1}{2}}AB^{\frac{-1}{2}}\|.$$
$R(A, B)$ was defined in \cite{and}, and many nice properties
of $R(A,B)$ were shown as follows: If $A,B, C\in\mathcal{P}$, then

\begin{align*}
&(i)~ R(A,C) \leq R(A,B)R(B,C) \text{ (triangle inequality) }\\
&(ii)~ R(A,B) \geq 1,\text{ and }R(A,B) = 1 \text{ iff } A = B\\
&(iii)~\|A-B\|\leq (R(A,B)-1)\|A\|.
\end{align*}

The Thompson distance $d(A, B)$
on the convex cone $\Omega$ of positive definite operators is defined by
$$d(A,B )  = \log  R(A, B)=\max\{\log r(A^{-1}B) , \log r(B^{-1}A) \},$$
see \cite{and, cor, nus}. we know that $\Omega$
is a complete metric space with respect to this metric and the corresponding metric topology
on $\Omega$ agree with the relative norm topology.

As a basic inequality with respect to the metric, the following inequality for a weighted geometric mean of two operators hold \cite{and, cor}:
$$ d(A_{1} \sharp_{\nu} A_{2},B_{1}\sharp_{\nu} B_{2})\leq (1-\nu)d(A_{1},B_{1})+\nu d(A_{2},B_{2})$$
         for $A_{1},A_{2},B_{1},B_{2} \in \Omega$ and $\nu \in (0,1)$.

Let $A$ and $B$ be two positive definite operators on a Hilbert space $H$ and $\Phi$ be a unital positive linear map on $B(H)$.
Ando \cite[Theorem 3]{ando1} showed the
following property of a positive linear map in
connection with the operator geometric mean.
\begin{equation}\label{2.1}
\Phi(A\sharp B)\leq \Phi(A)\sharp \Phi(B).
\end{equation}
Inequality \eqref{2.1} is extended to an operator mean $\sigma$ in Kubo-Ando theory as follows:
\begin{equation*}
\Phi(A\sigma B)\leq \Phi(A)\sigma \Phi(B),
\end{equation*}
In particular for the weighted geometric mean, we have
\begin{equation}\label{2.2}
\Phi(A\sharp_\nu B)\leq \Phi(A)\sharp_\nu \Phi(B),
\end{equation}
where $\nu$ is a real number in $(0,1]$.
A complementary inequality to inequality \eqref{2.2}, is the following important inequality\cite{mic}:\\
Let $0<m_1 I\leq A\leq M_1 I$ and $0<m_2 I\leq B\leq M_2 I$
and $0<\nu\leq 1$, then
\begin{equation}\label{2.3}
K(h,\nu)\Phi(A)\sharp_\nu \Phi(B)\leq \Phi(A\sharp_\nu B),
\end{equation}
where $h=\frac{M_1M_2}{m_1m_2}$ and $K(h,\nu)$ is the generalization Kantorovich constant defined by
\[
K(h,\nu)=\frac{h^\nu-h}{(\nu-1)(h-1)}\left(\frac{(\nu-1)(h^\nu-1)}{\nu(h^\nu-h)}\right)^\nu.
\]

The special case, $K(h,2)=K(h,-1)=\frac{(1+h)^2}{4h}$ is called the Kantorovich constant.
The generalization Kantorovich constant $K(h,\nu)$ has the following properties:
\begin{align*}
&K(h,\nu)=K(h,1-\nu)\\
0&<K(h,\nu)\leq 1 \text{ for } 0<\nu\leq 1,
\end{align*}
and $K(h,\nu)$ is decreasing for $\nu\leq \frac{1}{2}$ and increasing for $\nu> \frac{1}{2}$ therefore for all
$\nu\in \mathbb{R}$, $K(h)=K(h,\frac{1}{2})=\frac{2h^\frac{1}{4}}{1+h^\frac{1}{2}}\leq K(h,\nu)$.

For some fundamental results on complementary inequalities to famed inequalities, we would like to refer
the readers to \cite{mic, pec}. Afterwards, several complementary inequalities to
celebrated inequalities was discussed by many mathematicians.
For more information and some recent results on this topic; see \cite{and, fuj, tom, yam, yam1}.

\section{A complementary inequality to weighted geometric mean}\vskip 2mm

First, we state another complementary inequality to \eqref{2.2} with respect to $R(A,B)$ as follows:

\begin{theorem}\label{t1}
Let $A$ and $B$ be two positive definite operators in $\mathcal{P}$, then
\begin{equation}\label{2.4}
K\left(R(A,B)^2,\nu\right)\Phi(A)\sharp_\nu\Phi(B)\leq \Phi(A\sharp_\nu B).
\end{equation}
\end{theorem}

\begin{proof}
we know that
\begin{equation}\label{2.5}
\dfrac{1}{R(A,B)}A\leq B \leq R(A,B)A.
\end{equation}
Define a linear map $\Psi$ on $H$ by	

$$ \Psi(X)=\Phi(A)^{-\frac{1}{2}}\Phi(A^{\frac{1}{2}}XA^{\frac{1}{2}})\Phi(A)^{-\frac{1}{2}}.$$

Then $\Psi$ is a unital positive linear map. From \eqref{2.5}
\begin{align}\label{2.6}
m=\dfrac{1}{R(A,B)}\leq A^{-\frac{1}{2}}BA^{-\frac{1}{2}}\leq R(A,B)=M.
\end{align}
Using \eqref{2.3} for $\Psi $, and \eqref{2.6}, we get
\begin{equation}\label{2.8}
K(h,\nu)\Psi(I)\sharp_\nu\Psi(X)\leq\Psi(I\sharp_\nu X)
\end{equation}
where $X=A^{-\frac{1}{2}}BA^{-\frac{1}{2}}$ and $h=\dfrac{_{M}}{m}=R^{2}(A,B)$.

From \eqref{2.8}, we have
$$K(h,\nu)\Psi(X)^\nu\leq\Psi(X^\nu)$$ or
\begin{align*}
&K(h,\nu)\Big(\Phi(A)^{-\frac{1}{2}}\Phi(B)\Phi(A)^{-\frac{1}{2}}\Big)^\nu\leq
\Phi(A)^{-\frac{1}{2}}\Phi\Big(A^{\frac{1}{2}}(A^{-\frac{1}{2}}BA^{-\frac{1}{2}})^\nu A^{\frac{1}{2}}\Big)\Phi(A)^{-\frac{1}{2}}\\
&K(h,\nu)\Phi(A)^{\frac{1}{2}}\Big(\Phi(A)^{-\frac{1}{2}}\Phi(B)\Phi(A)^{-\frac{1}{2}}\Big)^\nu\Phi(A)^{\frac{1}{2}}\leq
\Phi\Big(A^{\frac{1}{2}}(A^{-\frac{1}{2}}BA^{-\frac{1}{2}})^\nu A^{\frac{1}{2}}\Big).
\end{align*}
Therefore, we obtain the desired inequality \eqref{2.4}.
\end{proof}
The inequality \eqref{2.3} with $\nu=\frac{1}{2}$, becomes the following inequality.

\begin{equation}\label{2.9}
\frac{2R(A,B)^\frac{1}{2}}{1+R(A,B)}\Phi(A)\sharp\Phi(B)\leq \Phi(A\sharp B).
\end{equation}

Next, we recall the Kadison's Schwarz inequalities

\begin{equation}\label{2.9.1}
\Phi(A^2)\geq \Phi(A)^2,~~\Phi(A^{-1})\geq \Phi(A)^{-1}
\end{equation}
and two complementary inequalities to them, whenever $0<m\leq A\leq M$:
\begin{align}\label{2.9.2}
\frac{(m+M)^2}{4mM}\Phi(A)^2\geq \Phi(A^2),~~\frac{(m+M)^2}{4mM}\Phi(A)^{-1}\geq \Phi(A^{-1}).
\end{align}
The following inequality unifies Kadison's Schwarz inequalities into a single form.
\begin{equation}\label{2.10}
\Phi(BA^{-1}B)\geq\Phi(B)\Phi(A)^{-1}\Phi(B).
\end{equation}
If $0<m\leq A, B\leq M$, then the following inequality is a complementary inequality to \eqref{2.10}
\begin{equation}\label{2.11}
\frac{(m+M)^2}{4mM}\Phi(B)\Phi(A)^{-1}\Phi(B)\geq \Phi(BA^{-1}B).
\end{equation}
The similar to the proof of Theorem\ref{t1}, we give another complementary inequality to \eqref{2.10} with respect to $R(A,B)$ as follows:

\begin{equation}\label{2.12}
\frac{(1+R(A,B)^2)^2}{4R(A,B)}\Phi(B)\Phi(A)^{-1}\Phi(B)\geq \Phi(BA^{-1}B).
\end{equation}

To compare the inequalities \eqref{2.3} with \eqref{2.9} and \eqref{2.11} with \eqref{2.12}, the following examples
show that neither Kantorovich constants in \eqref{2.3} and \eqref{2.11} nor Kantorovich constants in \eqref{2.9} and \eqref{2.12} are uniformly
better than the other.

\begin{example}
Let $A=
\begin{bmatrix}
  2 & 0 \\
  0 & \frac{1}{3} \\
\end{bmatrix}$
and $B=
\begin{bmatrix}
  4 & 0 \\
  0 & \frac{1}{2} \\
\end{bmatrix}$.
Clearly, $m_1I=\frac{1}{3}I\leq A\leq 2I=M_1I$ and $m_2I=\frac{1}{2}I\leq B\leq 4I=M_2I$. Then
$A^{-1}B=
\begin{bmatrix}
  2 & 0 \\
  0 & \frac{3}{2} \\
\end{bmatrix}$
and
$B^{-1}A=
\begin{bmatrix}
  \frac{1}{2} & 0 \\
  0 & \frac{2}{3} \\
\end{bmatrix}$,
therefore $R(A,B)^2=4\leq h=\frac{M_1M_2}{m_1m_2}=48$.
Consequently, $K(R(A,B)^2, \frac{1}{2})\geq K(h, \frac{1}{2})$ and $K(R(A,B)^2, 2)\leq K(h, 2)$.

\end{example}

\begin{example}
Let $C=
\begin{bmatrix}
  2 & 1 \\
  1 & 1 \\
\end{bmatrix}$
and $D=
\begin{bmatrix}
  1 & 0 \\
  0 & 2 \\
\end{bmatrix}$.
Clearly, $m_1I=\frac{3-\sqrt{5}}{2}I\leq C\leq \frac{3+\sqrt{5}}{2}I=M_1I$ and $m_2I=I\leq D\leq 2I=M_2I$. Then
$C^{-1}D=
\begin{bmatrix}
  1 & -2 \\
  -1 & 4 \\
\end{bmatrix}$
and
$D^{-1}C=
\begin{bmatrix}
  2 & 1 \\
  \frac{1}{2} & \frac{1}{2} \\
\end{bmatrix}$,
therefore $R(C,D)^2=\left(\frac{5+\sqrt{17}}{2}\right)^2\geq h=\frac{M_1M_2}{m_1m_2}=\frac{(3+\sqrt{5})^2}{2}$.
Consequently, $K(R(C,D)^2, \frac{1}{2})\leq K(h, \frac{1}{2})$ and $K(R(C,D)^2, 2)\geq K(h, 2)$.

\end{example}

\begin{theorem}\label{t2}
If $f$ is an operator convex function on $(0,\infty)$ and $\Phi$ is a unital positive linear map on $B(H)$. Then
\begin{align}\label{2.13}
K(h,2)f(\Phi(A))\geq \Phi(f(A))\geq f(\Phi(A))
\end{align}
for every positive definite operators $A$, where $h=R(A,B)^2$, or $h=\frac{M}{n}$ whenever $0<m\leq A\leq M$.
\end{theorem}

\begin{proof}
It is known that every operator convex function
$f$ on $(0,\infty)$ has a special integral representation as follows:
\begin{align}\label{2.14}
f(t)=\alpha+\beta t+\gamma t^2+\int_0^\infty\frac{\lambda t^2}{\lambda+t}d\mu(\lambda),
\end{align}
where $\alpha, \beta$ are real numbers, $\gamma \geq 0$, and $\mu$ is
a positive finite measure. Thus
\begin{align*}
\Phi(f(A))=\alpha1_K+\beta\Phi(A)+\gamma \Phi(A^2)+\int_0^\infty\Phi(\lambda A^2(\lambda+A)^{-1})d\mu(\lambda),
\end{align*}
and
\begin{align*}
f(\Phi(A))=\alpha1_K+\beta\Phi(A)+\gamma \Phi(A)^2+\int_0^\infty\lambda \Phi(A)^2(\lambda+\Phi(A))^{-1}d\mu(\lambda).
\end{align*}
By \eqref{2.10}, we have

\begin{align*}
\Phi(\lambda A^2(\lambda+A)^{-1})&=\lambda\Phi( A^2(\lambda+A)^{-1})=\lambda\Phi( A(\lambda+A)^{-1}A)\\
&\geq \lambda\Phi(A)(\Phi(\lambda+A))^{-1}\Phi(A)=\lambda\Phi(A)(\lambda+\Phi(A))^{-1}\Phi(A)\\
&=\lambda \Phi(A)^2(\lambda+\Phi(A))^{-1}.
\end{align*}
Therefore $\Phi(f(A))\geq f(\Phi(A))$, since $\Phi(A^2)\geq\Phi(A)^2$.

On the other hand, from \eqref{2.11} or \eqref{2.12}, we get
\begin{align*}
&\Phi(\lambda A^2(\lambda+A)^{-1})=\lambda\Phi( A^2(\lambda+A)^{-1})=\lambda\Phi( A(\lambda+A)^{-1}A)\\
&\leq K(h,2) \lambda\Phi(A)(\Phi(\lambda+A))^{-1}\Phi(A)=K(h,2)\lambda\Phi(A)(\lambda+\Phi(A))^{-1}\Phi(A)\\
&=K(h,2)\lambda \Phi(A)^2(\lambda+\Phi(A))^{-1}.
\end{align*}
Therefore $\Phi(f(A))\leq K(h,2) f(\Phi(A))$, since $\Phi(A^2)\leq K(h,2)\Phi(A)^2$.
\end{proof}

\section{The power means and the Karcher means}

The Karcher mean, also called the Riemannian mean, has long
been of interest in the field of differential geometry. Recently it has been used in a diverse variety of settings:
 diffusion tensors in medical imaging and radar, covariance matrices in statistics, kernels in machine learning and
elasticity. Power means for positive definite matrices and operators have been introduced in \cite{law2, lim1}.
It is shown in \cite{bha, law3} that the Karcher mean and power means satisfy all ten properties stated in \cite{ando}.
The Karcher mean and power means have recently become an important tool
for studying of positive definite operators and an interesting subject for matrix analysts and operator theorists.
We would like to refer the reader to \cite{law3,law2,lim1,lim2,moa,pal,yam1} and references therein for more information.

Geometric and power means of two operators
can be extended over more than $3$‐operators via the solution of operator
equations as follows. Let $n$ be a natural number, and let $\triangle_{n}$ be a set of all probability
vectors, i.e.,
\begin{equation*}
\triangle_{n}=\{ \omega= (w_{1}, \dots, w_{n})\in(0, 1)^{n}| \sum_{i=1}^{n}w_{i}=1\}.
\end{equation*}

Let $\mathbb{A}= (A_{1},\dots, A_{n}) \in \mathcal{P}^{n}$ and $\omega=(\mathrm{w}_{1},\dots, w_{n}) \in \triangle_{n} $.
Then the weighted Karcher mean $\Lambda( \omega; \mathbb{A})$ is defined by a unique
positive solution of the following operator equation;

\begin{equation}\label{3.1}
\sum_{i=1}^{n}\mathrm{w}_{i}\log(X^\frac{-1}{2}A_{i}X^\frac{-1}{2})=0.
\end{equation}
The weighted power mean $P_t( \omega; \mathbb{A})$ is defined by a unique
positive solution of the following operator equation;

\begin{equation*}
I=\sum_{i=1}^{n}\mathrm{w}_{i}(X^\frac{-1}{2}A_{i}X^\frac{-1}{2})^t \text{ for } t\in (0,1],
\end{equation*}
or equivalently
\begin{equation}\label{3.2}
X=\sum_{i=1}^{n}\mathrm{w}_{i}(X\sharp_t A_i) \text{ for } t\in (0,1].
\end{equation}

For $t\in [-1,0)$, it is defined by
\[
P_t( \omega; \mathbb{A})=(P_{-t}( \omega; \mathbb{A}^{-1}))^{-1},
\]
where $\mathbb{A}^{-1}= (A_{1}^{-1},\dots, A_{n}^{-1})$.

Lawson and Lim in \cite[Corollary 6.7]{law2} gave an important connection between the weighted Karcher means and
the weighted power means in strong operator topology, as follows:
\begin{equation}\label{3.3}
\Lambda(\omega, \mathbb{A})=\lim_{t\rightarrow 0}P_t(\omega, \mathbb{A}).
\end{equation}
Moreover, they also showed that the following property holds for these means:
\begin{equation}\label{3.4}
P_{-t}( \omega; \mathbb{A})\leq \Lambda( \omega; \mathbb{A})\leq P_{t}( \omega; \mathbb{A})\text{ for all }t\in(0,1].
\end{equation}

Let $\Phi$ be a unital positive linear map on $B(H)$. In \cite{lim1} is proved that if $t\in (0,1]$,
then
\begin{equation}\label{3.5}
\Phi(P_t( \omega; \mathbb{A}))\leq P_t( \omega; \Phi(\mathbb{A})),
\end{equation}
where $\Phi(\mathbb{A})=(\Phi(A_{1}),\dots, \Phi(A_{n}))$.

In the following Theorem, we give a reverse inequality to the inequality \eqref{3.5}.

\begin{theorem}\label{t3}
Let $\Phi$ be a unital positive linear map on $B(H)$ and $n\geq2$ be a positive integer. If $0< t\leq 1$ and $(A_1, \dots , A_n)\in \mathcal{P}^n$.
Then
\begin{equation}\label{3.6}
\Phi(P_t( \omega; \mathbb{A}))\geq K\left(h_0, \frac{1}{2}\right)^\frac{1}{t}P_t( \omega;\Phi(\mathbb{A})),
\end{equation}
where $h_0= \max\limits_{1\leq i,j\leq n} R^2(A_i,A_j)$.
\end{theorem}

\begin{proof}
Let $t\in (0,1]$ and $X_t=P_t( \omega; \mathbb{A})$. Then
$X_t=\sum_{i=1}^n \omega_i(X_t\sharp_t A_i)$, by the operator equation \eqref{3.2}.
From inequality \eqref{2.4}, we have
\begin{equation}\label{3.6.1}
K\left(R^2(X_t,A_i), t\right)(\Phi(X_t)\sharp_t \Phi(A_i))\leq \Phi(X_t\sharp_t A_i).
\end{equation}
In the proof of proposition 3.5 of \cite{law2}, was shown that
\begin{equation*}
d(P_t( \omega; \mathbb{A}), A_j)\leq \max\limits_{1\leq i,j\leq n} d(A_i,A_j),
\end{equation*}
Where $d$ is the Thompson metric. This implies that
\begin{equation*}
R^2(X_t, A_j)=R^2(P_t( \omega; \mathbb{A}), A_j)\leq \max\limits_{1\leq i,j\leq n} R^2(A_i,A_j)=h_0.
\end{equation*}
Since the function $K(h, \frac{1}{2})=\frac{2h^\frac{1}{4}}{1+h^\frac{1}{2}}$ is decreasing for $h\geq1$, Thus
\begin{equation}\label{3.6.2}
K\left(h_0, \frac{1}{2}\right)\leq K\left(R^2(X_t,A_i), \frac{1}{2}\right)\leq K\left(R^2(X_t,A_i), t\right)
\end{equation}
Define
$f(X)=\sum_{i=1}^n \omega_i(X\sharp_t \Phi(A_i))$. Then $\lim_{n\rightarrow\infty}f^n(X)=P_t( \omega; \mathbb{A})$ for any $X>0$.
Let $\Phi$ be a unital positive linear map on $B(H)$. From relations \eqref{3.6.1} and \eqref{3.6.2}, we get
\begin{align*}
\Phi(X_t)&=\sum_{i=1}^n \omega_i\Phi(X_t\sharp_t A_i)\geq \sum_{i=1}^n \omega_i K\left(R^2(X_t,A_i), t\right)(\Phi(X_t)\sharp_t \Phi(A_i))\\
&\geq \sum_{i=1}^n \omega_i K\left(R^2(X_t,A_i), \frac{1}{2}\right)(\Phi(X_t)\sharp_t \Phi(A_i))\\
&\geq\sum_{i=1}^n \omega_i K\left(h_0, \frac{1}{2}\right)(\Phi(X_t)\sharp_t \Phi(A_i))=K\left(h_0, \frac{1}{2}\right)f(\Phi(X_t)).
\end{align*}
Since $f$ is an increasing function, we have
\begin{align*}
f\left(K\left(h_0, \frac{1}{2}\right)^{-1}\Phi(X_t)\right)&=\sum_{i=1}^n \omega_i \left(K\left(h_0, \frac{1}{2}\right)^{-1}\Phi(X_t)\sharp_t \Phi(A_i)\right)\\
&=K\left(h_0, \frac{1}{2}\right)^{t-1}\sum_{i=1}^n \omega_i (\Phi(X_t)\sharp_t \Phi(A_i))\\
&=K\left(h_0, \frac{1}{2}\right)^{t-1}f(\Phi(X_t))\geq f^2(\Phi(X_t)).
\end{align*}
Therefore
\begin{align*}
\Phi(X_t)\geq K\left(h_0, \frac{1}{2}\right)f(\Phi(X_t))\geq K\left(h_0, \frac{1}{2}\right)^{1+(1-t)}f^2(\Phi(X_t)).
\end{align*}
Consequently
\begin{align*}
\Phi(X_t)\geq K\left(h_0, \frac{1}{2}\right)f(\Phi(X_t))\geq K\left(h_0, \frac{1}{2}\right)^{1+(1-t)+(1-t)^2+\dots+(1-t)^{n-1}}f^n(\Phi(X_t)).
\end{align*}
This implies that
\begin{align*}
\Phi(P_t(\omega; \mathbb{A}))=\Phi(X_t)\geq K\left(h_0, \frac{1}{2}\right)^\frac{1}{t}\lim_{n\rightarrow \infty}f^n(\Phi(X_t))=K\left(h_0, \frac{1}{2}\right)^\frac{1}{t}P_t(\omega; \Phi(\mathbb{A})).
\end{align*}
\end{proof}

Let $X$ be a positive definite operator on a Hilbert space $H$.
Then $m=\lambda_{min}(X)\leq X\leq \lambda_{max}(X)=M$, where $\lambda_{min}(X)(\text{resp. }\lambda_{max}(X))$ is the minimum (resp. maximum) of the specrum of $X$.

\begin{theorem}\label{t4}
Let $\Phi$ be a unital positive linear map on $B(H)$ and $n\geq2$ be a positive integer. If $(A_1, \dots , A_n)\in \mathcal{P}^n$.
Then
\begin{equation}\label{3.7}
\Lambda( \omega;\Phi(\mathbb{A}))\geq\Phi(\Lambda( \omega; \mathbb{A}))\geq \frac{4\hbar}{(1+\hbar)^2}~\Lambda( \omega;\Phi(\mathbb{A})),
\end{equation}
where
$m_i=\lambda_{min}(A_i),~M_i=\lambda_{max}(A_i)$ and $\hbar= \max\limits_{1\leq i\leq n} \frac{M_i}{m_i}$.
\end{theorem}

\begin{proof}
From relations \eqref{3.3} and \eqref{3.5}, we get
\begin{equation}\label{3.8}
\Phi(\Lambda( \omega; \mathbb{A}))\leq\Phi(P_t( \omega; \mathbb{A}))\leq P_t(\omega;\Phi(\mathbb{A})).
\end{equation}
If $t\rightarrow 0$ in \eqref{3.8}, then
\begin{equation}\label{3.9}
\Phi(\Lambda( \omega; \mathbb{A}))\leq \Lambda(\omega;\Phi(\mathbb{A})).
\end{equation}

Using \eqref{2.9.2}, we obtain the following inequalities for $\mathbb{A}=(A_1,A_2,\dots,A_n)$
\begin{equation}\label{3.9.1}
\Phi(\mathbb{A})^{-1}\leq\Phi(\mathbb{A}^{-1})\leq  K(\hbar,2)\Phi(\mathbb{A})^{-1}=\Phi(K^{-1}(\hbar,2)\mathbb{A})^{-1}.
\end{equation}

Let $-1\leq t<0$, utilizing  \eqref{3.5} and \eqref{3.9.1}, we have
\begin{align}\label{3.10}
\Phi(P_t( \omega; \mathbb{A}))&=\Phi\left(P_{-t}( \omega; \mathbb{A}^{-1}\right)^{-1})\geq\left(\Phi(P_{-t}( \omega; \mathbb{A}^{-1}))\right)^{-1}\notag\\
&\geq\left(P_{-t}( \omega; \Phi(\mathbb{A}^{-1}))\right)^{-1}\geq\left(P_{-t}( \omega; (\Phi(K^{-1}(\hbar, 2)\mathbb{A)})^{-1})\right)^{-1}\notag\\
&\geq P_{t}( \omega; \Phi(K^{-1}(\hbar, 2)\mathbb{A)})=K^{-1}(\hbar, 2)P_{t}( \omega; \Phi(\mathbb{A)}).
\end{align}
From relation \eqref{3.3} and \eqref{3.10}, we obtain
\begin{equation}\label{3.11}
K^{-1}(\hbar, 2)P_{-t}( \omega; \Phi(\mathbb{A)})\leq \Phi(P_{-t}( \omega; \mathbb{A}))\leq\Phi(\Lambda( \omega; \mathbb{A}))\leq \Lambda(\omega;\Phi(\mathbb{A})).
\end{equation}
If $t\rightarrow 0$ in \eqref{3.11}, then
\begin{equation*}
K^{-1}(\hbar, 2)\Lambda(\omega;\Phi(\mathbb{A}))\leq\Phi(\Lambda( \omega; \mathbb{A}))\leq\Lambda(\omega;\Phi(\mathbb{A})).
\end{equation*}

\end{proof}


\section{The geometric mean due to Ando-Li-Mathias}\vskip 2mm
Arithmetic and harmonic means of $n$-operators can be defined, easily. But the geometric mean case gives a lot of trouble
because the product of operators is non-commutative.
Although several geometric means
of $n$-operators are defined, they do not have some important properties, for example, permutation
invariant or monotonicity.
Some of mathematicians have long been interested to extend the geometric mean of two operators to $n$-operators case.
Ando, Lim and Mathias in \cite{ando} suggested a good definition of the geometric mean for extending it to
raise the number of positive semi-defined matrices. It is defined by a symmetric method and has many good properties.
They listed ten properties that a geometric mean of $m$ matrices
should satisfy, and displayed that their mean possesses all of them.
Lawson and Lim in \cite{law3} have shown that $G$ has all the ten properties. Other ideas of geometric
mean with all the ten properties have been suggested in \cite{bha1,bin, izu}.
In \cite{yam} Yamazaki pointed out that definition of the geometric mean by Ando. Li and Mathias can
be extended to Hilbert space operators.

The geometric mean $G(A_{1},A_{2},\dots ,A_{n})$ of any $n$-tuple
of positive definite operators $\mathbb{A}=(A_1,\dots,A_n)\in\mathcal{P}^n$ is defined by induction.

(i) $G(A_{1},A_{2})$ =$A_{1}\sharp A_{2}$

(ii) Assume that the geometric mean any $(n-1)$-tuple of operators is defined. Let
\[
G ((A_j)_{j\neq i })=G(A_{1},A_{2},\dots,A_{i-1},A_{i+1},\dots,A_{n}),
\]
and let sequences $\{\mathbb{A}_{i}^{(r)}\}  _{r=1}^{\infty}$ be $\mathbb{A}_{i}^{(1)}= A_{i}$  and
$ \mathbb{A}_{i}^{(r+1)}=G((\mathbb{A}_{j}^{(r)})_{j\neq i }) $.
If there  exists  $ \lim_{r\rightarrow\infty}{\mathbb{A}_{i}^{(r)}} $, and it does not depend on $i$.
Hence the geometric mean of $n$-operators is defined by
\begin{equation}\label{4.1}
 \lim_{r\rightarrow\infty}{\mathbb{A}_{i}^{(r)}} = G((\mathbb{A}))=G(A_{1},A_{2},\dots,A_{n}) \text{ for } i= 1,\dots,n.
\end{equation}
In \cite{ando}, Ando et al. showed there exists this limit and \eqref{4.1} is uniformly convergence.

Furthermore, for $\mathbb{A}=(A_1,\dots,A_n)$ and $\mathbb{B}=(B_1,\dots,B_n)\in\mathcal{P}^n$, the following important inequality holds
\begin{equation}\label{4.2}
R(G(\mathbb{A}), G(\mathbb{B}))\leq \left(\prod_{i=1}^n R(A_i,B_i)\right)^\frac{1}{n}.
\end{equation}
Particular,
\begin{equation}\label{4.3}
R(\mathbb{A}_i^{(2)}, \mathbb{A}_k^{(2)})=R(G((A_j)_{j\neq i}), G((A_j)_{j\neq k}))\leq R(A_i, A_k)^\frac{1}{n-1}
\end{equation}
holds.

Yamazaki in \cite{yam} also obtained a converse of the arithmetic-geometric mean inequality
of $n$-operators via Kantorovich constant. Soon after, Fujii el al. \cite{fuj} also proved a stronger
reverse inequality of the weighted arithmetic and geometric means due to Lawson and
Lim of $n$-operators by the Kantorovich inequality.

Let $\mathbb{A}=(A_1,\dots,A_n)\in\mathcal{P}^n$, $\Phi$ be a unital positive linear map on $B(H)$. Then $\Phi(A_{1}),\Phi(A_{2}),\dots,\Phi(A_{n})$ are $n$
positive definite operators in $\mathcal{P}$.
We consider the geometric mean $G(\Phi(\mathbb{A}))$, where $\Phi(\mathbb{A})=(\Phi(A_{1}),\Phi(A_{2}),\dots,\Phi(A_{n}))$,
as follows:
\begin{align*}
\Phi^{(1)}(\mathbb{A}_{i})&=\Phi(A_i) \text{ for } i=1,\dots,n,
\end{align*}
and for  $r\geq 1,  ~  i=1,\dots,n$
\begin{align*}
\Phi^{(r+1)}(\mathbb{A}_{i})&=G((\Phi^{(r)}(\mathbb{A}_{i}))_{j\neq i})\\
&=G(\Phi^{(r)}(\mathbb{A}_{1}),\dots,\Phi^{(r)}(\mathbb{A}_{i-1}), \Phi^{(r)}(\mathbb{A}_{i+1}),\dots, \Phi^{(r)}(\mathbb{A}_{n})).
\end{align*}
Finally,
\begin{align*}
G(\Phi(\mathbb{A}))=G((\Phi(A_{1}),\Phi(A_{2}),\dots,\Phi(A_{n}))) =\lim _{r\rightarrow\infty} \Phi^{(r)}(\mathbb{A}_{i}).
\end{align*}

\begin{theorem}\label{t5}
Let $n\geq2$ be a positive integer, and $(A_1, \dots , A_n)\in \mathcal{P}^n$.
Then
\begin{align}\label{4.4}
\Phi(G(\mathbb{A}))\geq\left(\frac{2h_1^{\frac{1}{2}}}{1+h_1}\right)^{n-1}G(\Phi(\mathbb{A}))
\end{align}
\end{theorem}
where $h_1=\max\limits_{1\leq i,j\leq n} R(A_i, A_j)$.
\begin{proof}

First, for $r\geq1$, we put $h_r=\max\limits_{1\leq i,j\leq n} R(\mathbb{A}_i^{(r)}, \mathbb{A}_j^{(r)})$ and $K_{r}=K(h^2_{r},\frac{1}{2})$.
By \eqref{4.3}, we have
\[
 1\leq h_{r} \leq h_{r-1}^{\frac{1}{n-1}} \leq\dots\leq h_{1}^{(\frac{1}{n-1})^{r}}.
\]
Concavity the function $ f(t)=t^{\alpha} $ for $0< \alpha\leq 1$  implies that
\begin{equation*}
\dfrac {1+t^{\alpha}}{2}\leq \left(\dfrac{1+t}{2}\right)^{\alpha},
\text{ or } \dfrac{2t^{\frac{\alpha}{2}}}{1+t^{\alpha}}\geq\left(\dfrac{2t^\frac{1}{2}}{1+t}\right)^{\alpha}.
\end{equation*}
Since the function $K(t)=\dfrac{2t^\frac{1}{2}}{1+t}$ is a decreasing function for $t\geq 1$, thus
\begin{equation}\label{4.5}
K_{r}=\dfrac{2h_{r}^\frac{1}{2}}{1+h_{r}}\geq\dfrac{2h_{1}^{\frac{1}{2}(\frac{1}{n-1})^r}}{1+h_{1}^{(\frac{1}{n-1})^r}}
\geq\left(\dfrac{2h_{1}^\frac{1}{2}}{1+h_{1}}\right)^{(\frac{1}{n-1})^r}=K_{1}^{(\frac{1}{n-1})^r}.
\end{equation}

Now, we will prove \eqref{4.4} by induction on $n$.
For $n=2$, from inequality \eqref{2.9}, we have
\begin{equation*}
\Phi(A_{1}\sharp A_{2}) \geq K_{1}(\Phi(A_{1}) \sharp \Phi( A_{2})).
\end{equation*}
For $n=3$. A simple calculation shows that
\begin{align*}
\Phi\Big(G(A_{1}\sharp A_{3},A_{1}\sharp A_{2})\Big)&=
\Phi\Big((A_{1}\sharp A_{3})\sharp(A_{1}\sharp A_{2} )\Big)
\geq K_{2}\Big[\Phi\Big(A_{1}\sharp A_{3}\Big)\sharp\Phi\Big(A_{1}\sharp A_{2}\Big)\Big]\\
&\geq K_{2}\Big[K_{1}\Big(\Phi(A_{1})\sharp \Phi(A_{3})\Big)\sharp K_{1}\Big(\Phi(A_{1})\sharp\Phi(A_{2})\Big)\Big]\\
&=K_{2}K_{1}\Big[\Big(\Phi(A_{1})\sharp\Phi(A_{3})\Big)\sharp\Big(\Phi(A_{1})\sharp(\Phi(A_{2})\Big)\Big].\\
&=K_{2}K_{1} G\Big(\Phi(A_{1})\sharp\Phi(A_{3}),\Phi(A_{1})\sharp\Phi(A_{2})\Big).
\end{align*}
Therefore $\Phi(A_{1}^{(3)})\geq K_{2}K_{1} \Phi^{(3)}(A_{1})$.
Hence for $r\geq1$, we obtain
\begin{equation*}
\Phi(A_{i}^{(r)}) \geq (K_{r-1}K_{r-2}\dots K_{1}) \Phi^{(r)}(A_{i}), \text{ for } i=1,2,3.
\end{equation*}

Assume that Theorem \ref{t2} holds for $n-1$. We prove the case $n$.
From the induction hypothesis, we have
\begin{align*}
\Phi(A_i^{(r)})&=\Phi(G ((A_j^{(r-1)})_{j\neq i }))=\Phi\left(G\big(A_{1}^{(r-1)},A_{2}^{(r-1)},\dots,A_{i+1}^{(r-1)},\dots,A_{n}^{(r-1)}\big)\right)\\\
&\geq K_{r-1}^{n-2}G\big(\Phi(A_1^{(r-1)}),\dots,\Phi(A_{i+1}^{(r-1)}),\dots,\Phi(A_n^{(r-1)})\big)\\
&\geq K_{r-1}^{n-2}G\big(K_{r-2}^{n-2}\Phi(A_1^{(r-1)}),\dots,K_{r-2}^{n-2}\Phi(A_{i+1}^{(r-1)}),\dots,K_{r-2}^{n-2}\Phi(A_n^{(r-1)})\big)\\
&\geq K_{r-1}^{n-2}K_{r-2}^{n-2}G\big(\Phi(A_1^{(r-1)}),\dots,\Phi(A_{i+1}^{(r-1)}),\dots,\Phi(A_n^{(r-1)})\big)\\
&= K_{r-1}^{n-2}K_{r-2}^{n-2}G\big(\Phi(G(A_1^{(r-2)}),\dots,\Phi(G(A_{i+1}^{(r-2)}),\dots,\Phi(G(A_n^{(r-2)})\big)\\
&\vdots\\
&\geq(K_{r-1}K_{r-2}\dots K_{1})^{n-2}G\big(\Phi(A_1),\dots,\Phi(A_{i+1}),\dots,\Phi(A_n)\big)\\
&=(K_{r-1}K_{r-2}\dots K_{1})^{n-2}G(\Phi(A_j)_{j\neq i}),
\end{align*}
therefore, for positive integer $i=1,2,\dots,n$, we deduce
\begin{equation}\label{4.6}
\Phi(A_i^{(r)})\geq(K_{r-1}K_{r-2}\dots K_{1})^{n-2}G(\Phi(A_j)_{j\neq i}).
\end{equation}
From \eqref{4.5}, we get
$$K_{r-1}K_{r-2}\dots K_{1}\geq K_{1}K_{1}^{\frac{1}{n-1}}\dots K_{1}^{(\frac{1}{n-1})^{r-1}}
=K_{1}^{1+\frac{1}{n-1}+\dots+(\frac{1}{n-1})^{r-1}}.$$
Consequently
\begin{equation}\label{4.7}
\liminf_{r\rightarrow \infty}K_{r-1}K_{r-2}\dots,K_{1}\geq K_{1}^{\frac{n-1}{n-2}},
\end{equation}
since $K_{1}^{1+\frac{1}{n-1}+\dots+(\frac{1}{n-1})^{r-1}}\rightarrow K_{1}^{\frac{n-1}{n-2}}$.
We know that
\begin{equation}\label{4.8}
\lim_{r\rightarrow\infty}\Phi(A_{i}^{(r)})=\Phi(\lim_{r\rightarrow\infty}A_{i}^{(r)})=\Phi(G(A_1, A_2,\dots,A_n))
\end{equation}
and
\begin{align}\label{4.9}
\lim _{r\rightarrow\infty} \Phi^{(r)}(A_{i})=G(\Phi(A_{1}),\Phi(A_{2}),\dots,\Phi(A_{n})).
\end{align}
From inequalities \eqref{4.6}, \eqref{4.7}, \eqref{4.8} and \eqref{4.9}, we obtain the desired inequality \eqref{4.4}.

\end{proof}

\begin{corollary}\label{c2}
Let $n\geq2$ be a positive integer and $(A_1, \dots , A_n)$ be a $n$-tuple in $\mathcal{P}^n$.
If $0<m_iI\leq A_i\leq M_iI$ for some scalers $0<m_i<M_i~(i=1,2,\dots,n)$.
Then
\begin{equation}\label{4.10}
\left(\frac{2\sqrt{M_0}}{1+M_0}\right)^{n-1}\left(\prod_{j=1}^{n}\langle A_{j}x,x\rangle\right)^{\frac{1}{n}}
\leq\langle G (A_{1},\dots,A_{n})x,x\rangle\leq \left(\prod_{j=1}^{n}\langle A_{j}x,x\rangle\right)^{\frac{1}{n}},
\end{equation}
where $M_0=\max\limits_{1\leq i,j\leq n}\{\frac{M_i}{m_j}\}$.
\end{corollary}

\begin{proof}
Let $(A_1, \dots , A_n)\in \mathcal{P}^n$. The relation $0<m_iI\leq A_i\leq M_iI$ implies that
\begin{align*}
R(A_i,A_j)=\max\{r(A_i^{-1}A_j), r(A_j^{-1}A_i)\}\leq \max\left\{\frac{M_j}{m_i}, \frac{M_i}{m_j}\right\}.
\end{align*}
Therefore
\[
h_0=\max_{1\leq i,j\leq n} R(A_i,A_j)\leq \max_{1\leq i,j\leq n}\left\{\frac{M_i}{m_j}\right\}=M_0.
\]
This implies that
\begin{equation}\label{4.11}
K(h^2_0,\frac{1}{2})\geq K(M^2_0, \frac{1}{2})=\frac{2\sqrt{M_0}}{1+M_0}.
\end{equation}

It can easily be verified by the iteration
argument from the two variable case that the following inequality holds
\begin{equation}\label{4.12}
\Phi(G(A_1, A_2,\dots,A_n))\leq G(\Phi(A_1), \Phi(A_2),\dots,\Phi(A_n)).
\end{equation}

Applying \eqref{4.4} and \eqref{4.12} to the positive linear functional $\Phi(A)=\langle Ax,x\rangle$, where $x$ is a unit vector in $H$,
we get the inequalities in \eqref{4.10} with Kantorovich constant $K(h^2_0,\frac{1}{2})$. The desired inequalities in \eqref{4.10} is obtained by \eqref{4.11}.
\end{proof}

In particular, if $0<mI\leq A_i\leq MI$, then $M_0=\frac{M}{m}$ and the relation \eqref{4.10} becomes
\begin{equation*}
\left(\frac{2\sqrt{mM}}{m+M}\right)^{n-1}\left(\prod_{j=1}^{n}\langle A_{j}x,x\rangle\right)^{\frac{1}{n}}
\leq\langle G (A_{1},\dots,A_{n})x,x\rangle\leq \left(\prod_{j=1}^{n}\langle A_{j}x,x\rangle\right)^{\frac{1}{n}}.
\end{equation*}

From inequalities \eqref{4.4} for $\Phi(A)=\langle Ax,x\rangle$ and as regards for positive linear operators
 $\|A\|=\sup\{\langle Ax,x\rangle :~\|x\|\leq 1\}$,
we give the following result:

\begin{corollary}\label{c3}
Let $n\geq2$ be a positive integer, $(A_1, \dots , A_n)\in \mathcal{P}^n$ and let $0<m_iI\leq A_i\leq M_iI$ for $i=1,2,\dots,n$ and for some scalers $0<m_i<M_i$.
Then
\begin{equation}\label{2.20}
\left(\frac{2\sqrt{M_0}}{1+M_0}\right)^{n-1}\prod_{j=1}^{n}\| A_{j}\|^{\frac{1}{n}}
\leq\| G (A_{1},\dots,A_{n})\|\leq \prod_{j=1}^{n}\| A_{j}\|^{\frac{1}{n}},
\end{equation}
where $M_0=\max\limits_{1\leq i,j\leq n}\{\frac{M_i}{m_j}\}$.
\end{corollary}


\end{document}